\documentclass{article}

\usepackage{fullpage}

\RequirePackage[OT1]{fontenc}
\RequirePackage{amsthm,amsmath,amssymb}
\RequirePackage[numbers]{natbib}
\RequirePackage[colorlinks,citecolor=blue,urlcolor=blue]{hyperref}
\usepackage{graphicx}

\numberwithin{equation}{section}
\theoremstyle{plain}

\usepackage[utf8]{inputenc} 
\usepackage[T1]{fontenc} 
\usepackage{url}  
\usepackage{booktabs} 
\usepackage{amsfonts} 
\usepackage{nicefrac} 
\usepackage{microtype} 
\usepackage{enumerate}

\usepackage{algorithm2e}

\usepackage{color}

\theoremstyle{plain}

\newtheorem{prop}{Proposition}
\newtheorem{theorem}{Theorem}

\theoremstyle{definition}
\newtheorem{defi}{Definition}

\theoremstyle{remark}

\title{Stable network inference in high-dimensional graphical model \\ using single-linkage}

\author{Emilie Devijver, Mélina Gallopin, Rémi Molinier}

\begin{document}
\maketitle

\begin{abstract}
Stability, akin to reproducibility, is crucial in statistical analysis. This paper examines the stability of sparse network inference in high-dimensional graphical models, where selected edges should remain consistent across different samples. Our study focuses on the Graphical Lasso and its decomposition into two steps, with the first step involving hierarchical clustering using single linkage.
We provide theoretical proof that single linkage is stable, evidenced by controlled distances between two dendrograms inferred from two samples. Practical experiments further illustrate the stability of the Graphical Lasso's various steps, including dendrograms, variable clusters, and final networks. Our results, validated through both theoretical analysis and practical experiments using simulated and real datasets, demonstrate that single linkage is more stable than other methods when a modular structure is present.
\end{abstract}


\section{Introduction}
\label{introduction}

Interpretability, reproducibility and stability have become central  challenges in statistics due to the recent advances in massive data and black-box models \cite{VDS_2024}. From a learning theory  perspective, stability is crucial for generalization \cite{bousquet2002} while in practice, stability is fundamental for interpretability.  This paper focuses on algorithmic stability, which pertains to  the robustness of a procedure against data perturbation: does a method provide the same results on two perturbed data sets? 
Data perturbation techniques such as jackknife, sub-sampling or bootstrap have been extensively studied both theoretically and practically to assess the stability of statistical methods. 
\cite{Philipp_2017} propose a framework for evaluating the stability of data set and predictive algorithms, concluding that even an inherently unstable method can appear stable over generated data if under the model.
Recent methods are  driven by the concept of stability \cite{yuBernoulli, Yu_2016}. Stability in prediction has been explored accross various  models, including random forests for  interaction studies \cite{basu2018}, bagging  \cite{bagging2024}, and feature selection \cite{pfitser2021}.

Network inference is a domain within statistics where stability is  particularly critical. When performing network inference on two data sets derived from the same model with a small sample size using classical methods, the resulting inferred networks are often markedly different. This variability arises  from the large number of parameters that need to be estimated  and the complexity of the optimization task involved. 
However, without stability, interpretability becomes challenging, which undermines one of the major advantages of graphical models. 
 This is especially pertinent in the context of regulatory networks derived from real omics data, where observations are typically limited \cite{Frazee2011, Krumsiek2011, MICHAILIDIS2013326}.
As a result, practitioners have often criticized the developed methods, opting instead to manually select an appropriate subset of variables to focus on. However, such external knowledge is not always available and could be enhanced by a deeper understanding of the data and the application of machine learning tools.

Among various tools, graphical models are popular for network inference, and particularly valued for their interpretability. Gaussian Graphical Models (GGMs) are famous for embodying the Markov property.  This property links the edges of the corresponding dependency graph between variables to the non-zero coefficients of the inverse covariance matrix. Hence, GGMs facilitate understanding the complex relationships among variables by translating statistical dependencies into a graphical representation, where each edge signifies a direct conditional dependency between variables.
 The Graphical Lasso \cite{Friedman2007, Yuan2007} is a classical estimator that provides a sparse inverse covariance matrix  $\Theta = \Sigma^{-1}$, solution of the following optimization problem: for a sample $(\mathbf{y}_1, \ldots, \mathbf{y}_n)$ coming from a random variable $\mathbf Y \sim \mathcal{N}_p(0,\Sigma)$, and its sample covariance estimate $S$,
\begin{align*}
 \hat{\Theta}^{GL}(\lambda;S) = \underset{\Theta}{\operatorname{argmax}} \left\{ \log \det \Theta - \text{tr}(S \Theta) - \lambda \|\Theta\|_1 \right\}
\end{align*}
over nonnegative definite matrices $\Theta$, and where $\lambda$ is a nonnegative tuning parameter. 
Many algorithms have been proposed for solving the Graphical Lasso problem, but we focus here on the decomposition given in \cite{Witten2011GM, Mazumder2012}:
\begin{enumerate}
 \item[Step 1] Identify the connected components of the undirected graph with adjacency matrix $A$ associated to the thresholded sample covariance;
 \item[Step 2] Perform Graphical Lasso with parameter $\lambda \geq 0$ on each connected component separately. 
\end{enumerate}
This approach has been practically used in \cite{Danaher2014, Hsieh2014} to reduce the number of parameters to estimate for a fixed level of regularization, thereby improving computational efficiency. In \cite{Tan2015}, it was demonstrated that identifying the connected components in the Graphical Lasso solution (first step) is equivalent to performing single linkage hierarchical clustering based on a similarity matrix derived from the absolute values of the elements of the sample covariance matrix $S$.
This decomposition allows flexibility in the choice of linkage in hierarchical clustering. \cite{Tan2015} switched to average linkage, critiquing the chain effect of single linkage, and selected a model with two clusters, inferring a sparser model within each module. Conversely, \cite{devijverGallopin} retained single linkage but provided non-asymptotic theoretical foundations for selecting the number of clusters.

In this paper, we argue that this decomposition into two steps enhances the stability of network inference. We experimentally illustrate this improvement and theoretically prove that single linkage is stable, whereas other classical linkages, such as average linkage, are not.

Several methods have been proposed to stabilize variable selection in GGMs, primarily based on resampling. In \cite{bachBoLasso, Meinshausen2006}, the authors suggest subsampling the observations, running a model on each sample, and retaining variables selected consistently across all or most samples. Both papers provide theoretical results that guarantee good performance asymptotically with increasing sample sizes. Building on \cite{bachBoLasso}, \cite{colby_2018} evaluate the stability and accuracy of gene regulatory network inference using bootstrap aggregation. Additionally, \cite{Haury}, drawing from \cite{bachBoLasso, Meinshausen2006}, focuses specifically on bootstrap sampling for network inference. 
More recently, \cite{chiquet_2023} proposed a score to measure the overall stability of the set of selected features, introducing a new calibration strategy for stability selection. In a broader context, \cite{LimYu} introduced \texttt{ESCV}, while \cite{Bar-Hen2016} proposed removing the most influential observations to achieve stable networks, akin to the jackknife method.

However, these methods require substantial computation because they rely on subsampling. Furthermore, large sample sizes are necessary to ensure good performance with subsampling techniques.

Our main idea is that \textbf{estimators can be stable by construction and do not necessarily require additional steps to achieve stability}. This intrinsic stability can lead to more efficient and robust network inference.

In this paper, we make the following contributions:
\begin{itemize}
 \item We derive theoretically the stability of the decomposition of the network into independent modules using hierarchical clustering;
 \item We show experimentally on simulated data and on real data sets the stability of the hierarchical clustering, of the subsequent clusters, and of the inferred network. 
\end{itemize}

The remainder of the paper is organized as follows. In Section \ref{modelMethod}, we introduce the main theoretical result, about the stability of the hierarchical clustering. 
Section \ref{simulations} investigates the numerical stability through several experiments on simulated and real dataset: study of the hierarchical clustering for several linkages, study of the considered clusters when selecting a model in the dendogram, and study of the inferred network.

\section{Theoretical result for the stability of the modular decomposition}
\label{modelMethod}
In this section, we are interested in the stability of the hierarchical clustering, in the sense that, if two samples are observed generated from the same distribution, we want to measure how close are the two dendograms provided by the hierarchical clustering.
Let $(\mathbf{y}_1, \ldots, \mathbf{y}_n)$ and $(\mathbf{y}_1, \ldots,\mathbf{y}_{i-1},\tilde{\mathbf{y}}_i, \mathbf{y}_{i+1} \ldots, \mathbf{y}_n)$ be two samples in $\mathbb{R}^{ p}$ from the same multivariate normal distribution with density $\phi_p(\mathbf{0}, \Sigma)$ where $\Sigma_{j,j} = 1$ for all $j \in \{1,\ldots,p\}$. We assume that observations are standardized, and we focus on empirical correlations matrices.

We start by the definition of a dendogram. 
\begin{defi}
A \emph{dendrogram} over $\{1,\ldots,p\}$ is an application $\theta : [0,\infty) \rightarrow \mathcal{P}(\{1,\ldots,p\})$, where $\mathcal{P}(\{1,\ldots,p\})$ is the set of all partitions of $\{1,\ldots,p\}$, such that
\begin{enumerate}
\item $\theta(0)$ is the partition with only singletons;
\item there exists $M>0$ such that for all $t>M$, $\theta(t)=\{1,\ldots,p\}$;
\item for every $t_1\leq t_2$, the partition $\theta(t_1)$ is a refinement of the partition $\theta(t_2)$; and
\item for all $t_0>0$, there exists $\epsilon>0$ such that for all $t\in[t_0,t_0+\epsilon]$, $\theta(t)=\theta(t_0)$.
\end{enumerate}
\end{defi}
In other words, $\theta$ defines a nested family of partitions of $\{1,\ldots,p\}$ which, according to the two other properties, starts with only singletons and ends with the whole space. The last condition ensures right-continuity, which allow the existence of some minimums (for example, in the definition of $\Psi$ that follows). We denote ${\Theta}_{p}$ the set of all dendrograms on $\{1,\ldots,p\}$.

We work here with ultrametrics, that are associated to dendrograms through a one-to-one mapping. 
Ultrametric spaces are metric spaces which satisfy a stronger type of triangle inequality. 
\begin{defi}
A metric space $(X,u)$ is called an \emph{ultrametric space} if, for all $(x,x',x'' )\in X^3$, 
\[\max(u(x,x'),u(x',x''))\geq u(x,x'').\] 
For a finite set $\{1,\ldots,p\}$, we denote $\mathcal{U}_p$ the set of all ultrametrics on $\{1,\ldots,p\}$. 
\end{defi}

Theorem 9 in \cite{CM10} gives a one to one correspondence 
$$
\Psi \colon {\Theta}_p \to \mathcal{U}_p
$$
where, for $\theta\in\Theta_p$, $u=\Psi(\theta)$ is the ultrametric on $\{1,\ldots,p\}$ defined for all $(x,y)\in \{1,\ldots,p\}^2$ by
$$
u(x,y)=\min \left\{ t\geq 0 \mid x\text{ and } y \text{ are in the same subset in the partition $\theta(t)$}\right\}.
$$

Note that $u=\Psi(\theta)$ is also, by definition, the \emph{cophenetic} distance associated to the dendogram $\theta$: $u(i,j)$ corresponds to the height at which stage $i$ and $j$ are merged together. We compare those cophenetic distances for two dendograms using the following distance. 
\begin{defi}\label{def:coph}
 The distance $d_{\text{coph}}$ is defined by, for two dendograms $\theta_{1},\theta_{2} \in \Theta_p$, and their associated ultrametrics $u_1=\Psi(\theta_1)$ and $u_2=\Psi(\theta_2)$,
 \begin{equation}\label{eq:coph}
 d_{\text{coph}}(\theta_1, \theta_2) = \max_{1\leq i,j \leq p} |u_{1}(i,j)-u_{2}(i,j)|.
\end{equation}
\end{defi}

The inverse $\theta=\Psi^{-1}(u)$ for $u\in\mathcal{U}_p$ is given, for $t\geq 0$, by $\theta(t)$ to be the partition obtained from the equivalence relation $\sim_{u,t}$ where, for $(x,y)\in \{1,\ldots,p\}^2$, 
$$
x\sim_{u,t}y\quad \Longleftrightarrow \quad u(x,y)\leq t.
$$

We will denote by $C_p$ the complete simple graph with $\{1,\ldots,p\}$ as set of vertices and a path in $C_p$ with $\nu$ vertices will be encoded by a map $\eta\colon\{1,2,\dots, \nu\} \to \{1,2,\dots, p\}$ where, for all $k\in\{1,2,\dots \nu\}$, $\eta(\nu)$ yields the $k$th vertex of the path.
We introduce in the next definition  the application $u_A$ and then show that it is an ultrametric on $\{1,\ldots, p\}$. 
\begin{defi}
Let $A \in \mathbb{R}^{p\times p}$ be a symmetric matrix with positive nondiagonal entries, and zeros on the diagonal. We define the following application.
\begin{align*}
u_A \colon \{1,\ldots,p\}^2 &\longrightarrow \mathbb{R} \\
(i,j) &\longmapsto
\begin{cases}
 0 \text{ if }i=j,\\
\displaystyle{ \underset{\eta \text{ a path from $i$ to $j$ in $C_p$}}{\min} \left\{\max_{k}(A_{\eta(k),\eta(k+1)}) \right\}} \text{ elsewhere.}
\end{cases}
\end{align*}
\end{defi}

\begin{prop}
\label{lem::uA}
Let $A \in \mathbb{R}^{p\times p}$ be a symmetric matrix with positive nondiagonal entries, and zeros on the diagonal. 
The application $u_A$ defines an ultrametric on $\{1,\ldots,p\}$.
\end{prop}
\begin{proof}
We check all the properties of an ultrametric. 
 
\textbf{Positive definiteness:} for all $(i,j)\in \{1,2,\dots,p\}^2$, $u_A(i,j)\geq 0$ and, by assumption, $u_A(i,j)=0$ if and only if $i=j$.
 
\textbf{Symmetric:} since $a$ is symmetric, $u_A$ is also symmetric.
 
\textbf{Strong triangle inequality:} let us prove that for all $(i,j,k)\in \{1,2,\dots,p\}^3$,
\[u_A(i,k)\leq \max(u_A(i,j),u_A(j,k)).\]
Fix $(i,j,k)\in\{1,2,\dots, p\}^3$ and consider two paths in the complete graph $C_p$
\begin{align*}
 \eta_1&\colon \{1,2,\dots, k_1\}\to\{1,2\dots,p\}, \\
 \eta_2&\colon \{1,2,\dots,k_2\}\to\{1,2,\dots,p\},
\end{align*}
such that 
\begin{align*}
 u_A(i,j)&=\max_{l\in\{1,\dots,k_1-1\}}\left(a_{\eta_1(l),\eta_1(l+1)}\right),\\
 u_A(j,k)&=\max_{l\in\{1,\dots,k_2-1\}}\left(a_{\eta_2(l),\eta_2(l+1)}\right).
\end{align*}
Then, if we set $k=k_1+k_2-1$ and $\eta_{1}\cdot\eta_2$ the path $\eta_1$ followed by $\eta_2$ (which is a path from $i$ to $k$ in $C_p$), then $\eta$ has $k$ vertices and we have
\begin{align*}
u_A(i,k) &= \min_{\eta} \max_{l\in\{1,\dots,k-1\}}\left(a_{\eta(l),\eta(l+1)}\right)\\
&\leq \max_{l\in\{1,\dots,k-1\}}\left(a_{\eta_{1}\cdot\eta_2(l),\eta_{1}\cdot\eta_2(l+1)}\right) \qquad \text{($\eta_{1}\cdot\eta_2$ is a path from $i$ to $k$)} \\
&=\max\left( \max_{l\in\{1,\dots,k_1-1\}}\left(a_{\eta_1(l),\eta_1(l+1)}\right), \max_{l\in\{1,\dots,k_2-1\}}\left(a_{\eta_2(l),\eta_2(l+1)}\right) \right)\\
&= \max\left(u_A(i,j),u_A(j,k)\right).
\end{align*}
Hence, $u_A$ is an ultrametric on $\{1,2,\dots,p\}$.
\end{proof}
For a symmetric matrix $A\in \mathbb{R}^p$ with positive nondiagonal entries and zeros on the diagonal, we denote by $\theta_A = \Psi^{-1}(u_A)$ the dendrogram associated to the ultrametric $u_A$. 

Remark that, if the matrix $A$ is associated to a distance $d$, the dendogram $\theta_A$ is exactly the one obtained by the single linkage hierarchical clustering with the distance $d$ (see \cite[Corollary 14]{CM10}). One can particularly use $A = \mathbf{1} - |S^1|$, with $S^1$ the sample covariance matrix and $\mathbf{1}$ corresponds to the matrix with $1$ for each coefficient, which is the one constructed in the first step of the Graphical Lasso \cite{Tan2015}. 





The next proposition gives a control on the  distance introduced in Definition \ref{def:coph}, between dendograms induced by two different matrices.

\begin{prop}
\label{2emeLemmeDeRemi}
Let $A_1,A_2\in\mathbb{R}^{p\times p}$ be two symmetric matrices with positive entries, and zeros on the diagonal.
Then 
\[d_{\text{coph}}(\theta_{A_1}, \theta_{A_2}) \leq \Vert A_1 - A_2 \Vert_{\max} \]
where $||.||_{\max}$ corresponds to the maximum element of the matrix.
\end{prop}
\begin{proof}
Set $m=\Vert A_1 -A_2\Vert_{\max}$.
Let $(i,j)\in \{1,2,\dots,p\}^2$ and let $\eta_2\colon \{1,2,\dots, K_2\}\to \{1,2,\dots p\}$ be a path such that
\[u_{A_2}(i,j)=\max_{k\in\{1,\dots,K_2-1\}} [A_2]_{\eta_2(k),\eta_2(k+1)}.\]
Then we have,
\begin{align*}
 u_{A_1}(i,j) &\leq \max_{k\in\{1,\dots,K_2-1\}} [A_1]_{\eta_2(k),\eta_2(k+1)} &\text{(by definition of $u_{A_1}$)}\\
 &\leq \max_{k\in\{1,\dots,K_2-1\}}\left(m+[A_2]_{\eta_2(k),\eta_2(k+1)}\right) &\text{(by definition of $m$)}\\
 &= m+u_{A_2}(i,j) &\text{(by choice of $\eta_2$).}
\end{align*}
Hence, $u_{A_1}(i,j)-u_{A_2}(i,j)\leq m$ and symmetrically, $u_{A_2}(i,j)-u_{A_1}(i,j)\leq m$. Therefore, for all $(i,j)\in \{1,2,\dots,p\}^2$, $|u_{A_1}(i,j)-u_{A_2}(i,j)|\leq m$. Thus, 
$$\max_{i,j} |u_{A_1}(i,j)-u_{A_2}(i,j)|\leq m.$$
\end{proof}

Finally, we control the stability of the dendogram constructed in the first step of the Graphical Lasso in the following proposition. 

\begin{prop}\label{prop:stat}
 Let two samples $(\mathbf y_1,\ldots,\mathbf y_n)$ and $(\mathbf y_1, \ldots, \tilde{\mathbf y}_i, \ldots, \mathbf y_n)$ where $\tilde{\mathbf y}_i \sim \mathbf y_i \sim \mathbf{Y}$ and are iid, and $S$ and $\tilde{S}$ the corresponding sample covariance matrices. Then, for $\alpha \in (0,1)$, with probability $1-\alpha$,

 \begin{align*}
 \|S - \tilde{S}\|_{\infty} \leq\frac{2p}{(n-1)\sqrt{\alpha}}. 
 \end{align*}
\end{prop}

\begin{proof}
For $1\leq j,k \leq n$,
$$[S - \tilde S]_{j,k} = \frac{1}{n-1} \left(y_{i,j} y_{i,k} - \tilde{y}_{i,j} \tilde{y}_{i,k}\right).$$

From \cite{NADARAJAH2016201}, we know the distribution function of $ Y_j Y_k$:
\begin{align*}
 f_{Y_jY_k}(z;\rho) &= \frac{1}{\pi \sqrt{1-\rho^2}} \exp\left(\frac{\rho z}{1-\rho^2}\right) K_0\left(\frac{\|z\|}{1-\rho^2}\right),
\end{align*}
where $\rho = \Sigma_{j,k}$ is the correlation between $Y_j$ and $Y_k$. 
Eventually, we can compute the first two moments of $[S - \tilde S]_{j,k}$:
\begin{align*}
E([S - \tilde S]_{j,k}) &= 0,\\
 Var([S - \tilde S]_{j,k}) &= Var\left(\frac{1}{n-1} \left(y_{i,j} y_{i,k} - \tilde{y}_{i,j} \tilde{y}_{i,k}\right)\right)\\
 &= \frac{2}{(n-1)^2} Var\left(y_{i,j} y_{i,k}\right)\\
 &= \frac{2}{(n-1)^2} (1+\Sigma_{j,k}^2),
\end{align*}
where  the second line comes from the independence of $y$ and $\tilde y$, and the last equality comes from \cite{Gaunt2022}, which characterizes the moments of the product of zero mean correlated normal random variables. 
%
We need to control $\|S - \tilde{S}\|_\infty$. 
The union bound gives:
\begin{align*}
\mathbb{P}(\|S - \tilde{S}\|_\infty \geq \eta) &=
\mathbb{P}(\max_{j,k}|S_{j,k} - \tilde{S}_{j,k}| \geq \eta)\\
&\leq \sum_{j,k} \mathbb{P}(|S - \tilde{S}|_{j,k} \geq \eta)\\
&= \sum_{j,k} \mathbb{P}(\frac{1}{n-1} \left|y_{i,j} y_{i,k} - \tilde{y}_{i,j} \tilde{y}_{i,k}\right|\geq \eta)\\
&\leq \sum_{j,k} 2 \cdot \frac{1+\Sigma_{k,j}^2}{(n-1)^2\eta^2},
\end{align*}
where the last inequality comes from Tchebychev's inequality. 
Using that the correlations are bounded by $1$, we get that
 \begin{align*}
\mathbb{P}(\|S - \tilde{S}\|_\infty \geq \eta) &\leq \frac{4 p^2}{(n-1)^2\eta^2}.
\end{align*}
\end{proof}
This result is derived for a fixed sample size. Non-asymptotically, we have a control on the difference between the two dendrograms. Moreover, note that we used the Tchebychev inequality, but there may exist tighter concentration inequality. 

Proposition \ref{prop:stat} gives a stability result about the ultrametric induced by the empirical correlation matrix. Using the one-to-one correspondence $\Psi$, it can also be interpreted in terms of stability of the induced dendrogram as explained in \citet[Section 3.5]{CM10}. It leads to the following theorem, which is our main theoretical contribution.

\begin{theorem}\label{thm}
 Let two samples $(\mathbf y_1,\ldots,\mathbf y_n)$ and $(\mathbf y_1, \ldots, \tilde{\mathbf y}_i, \ldots, \mathbf y_n)$ where $\tilde{\mathbf y}_i \sim \mathbf Y$ and are iid, and $S$ and $\tilde{S}$ the corresponding sample covariance matrices. Then, for $\alpha \in (0,1)$, with probability $1-\alpha$,
 \begin{align*}
 d_{coph}( |\theta_{\mathbf{1} - |S|}|, \theta_{\mathbf{1} - |\tilde{S}|}) \leq\frac{2p}{(n-1)\sqrt{\alpha}}.  
 \end{align*}
\end{theorem}

\begin{proof}
By Proposition \ref{2emeLemmeDeRemi} we have 
\[d_{coph}( |\theta_{\mathbf{1} - |S|}|, \theta_{\mathbf{1} - |\tilde{S}|}) \leq \left\Vert  \mathbf{1} - |S| - \left(\mathbf{1} - |\tilde{S}|\right) \right\Vert_{\max}=\left\Vert |S|-|\tilde{S}|\right\Vert_{\max}.\]
By the triangular inequality, the last term is less or equal to $\Vert S-\tilde{S}\Vert_{\max}$ and then Theorem \ref{thm} follows from Proposition \ref{prop:stat}.
\end{proof}

Asymptotically, the two collections of modules detected by the Graphical Lasso on two samples where only one observation differs, varying the regularization parameter $\lambda$, are the same. This means that the single linkage used in the first step of the Graphical Lasso is a good choice, with respect to the stability of the collection of models that is considered. 

Similarly to \citet[Remark 17]{CM10}, we can show that the complete linkage and the average linkage are unstable: small perturbations of the matrix $A$ may lead to large perturbations of the corresponding ultrametric.

\section{Experiments}
\label{simulations}
In this section, we evaluate in practice the stability of each step of the Graphical Lasso. 
First we provide the experimental design, describing the data generation process, and the two real dataset we are studying.
Then, we illustrate 1/ the stability of dendograms using hierarchical clustering, varying the linkage (illustrating exactly Theorem \ref{thm}), 2/ the stability of the clusters get by cutting the dendogram with some model selection criterion, 3/the stability of the network inference.

\subsection{Experimental design}
The design of the simulations is as follows. 
For a fixed structured covariance matrix $\Sigma$ with a block diagonal structure, we simulate $V$ samples of $n$ observations from a $p$-variate normal distribution with a  mean of zero and the structured covariance matrix $\Sigma$. We then compare the inferred networks pairwise, resulting in $V(V-1)/2$ comparisons. 
The number of variables is set to $p=100$, the sample size to $n=70$, and the number of samples to $V=17$ per fixed covariance matrix. We consider $R=5$ different covariance matrices (generated randomly, each with the same block decomposition), with the number of blocks in the diagonal matrix set to $K=15$, and each block containing  $6$ or $7$ variables.

Two real datasets are considered. \\
 The BRCA dataset is a gene expression dataset for patients with breast cancer, measured with RNA-Sequencing. The data are generated by the TCGA Research Network: \url{http://cancergenome.nih.gov/}, and downloaded from the  web portals \url{https://tcga-data.nci.nih.gov/tcga/} using the TCGA2STAT tool \cite{Wan2015}. We have $n=1212$ samples and $p=9191$ genes, but we focus on the $p=200$ most variable genes. We construct 17 batches of size 70, leading to $1190$ observations.
 \\
 Equities dataset includes stock market data  available in the R package \texttt{huge} and has been studied in \cite{Tan2015}. It contains closing prices of 452 stocks over 1258 trading days. We focus on the $p=200$ most variables stocks' close prices and we construct 17 patches of size 70.

\subsection{Stability of  hierarchical clustering: which linkage method?}
\label{hierarchicalClustering}
In this section, we validate the theoretical results obtained in Section \ref{modelMethod}. 
When applying hierarchical clustering, various linkage methods can be used.  We compare the performance in stability of the most well-known methods: average linkage (AL), complete linkage (CL), McQuitty linkage (ML), single linkage (SL), and Ward linkage (WL).
The measure used to compare dendograms is the distance introduced in Definition \ref{def:coph}, which we normalize to facilitate the analysis: for two matrices $A_1, A_2$,
and their associated dendograms $\theta_{A_1}, \theta_{A_2}$,

$$ d_{\text{coph}}^N(\theta_{A_1}, \theta_{A_2}) = \max_{1\leq i,j \leq p} \left|\frac{u_{A_1}(i,j)}{\max (u_{A_1})}-\frac{u_{A_2}(i,j)}{\max(u_{A_2})}\right|.$$

\begin{table}[t]

 \caption{Stability of dendograms using hierarchical clustering with several linkages on generated data, BRCA and equities. We display the normalized distance and its standard deviation in parenthesis. Best results are bolded. }
 \label{tab:linkage}
\centering
\begin{tabular}{c|ccccc}
\hline
 $\text{coph}_n$ &AL&CL&ML&SL&WL \\ \hline 
 Generated data & 0.53 (0.12) &0.72 (0.12)& 0.54 (0.12)& \textbf{0.33} (0.06) & 0.72 (0.14)\\
 BRCA & 0.81 (0.08)& 0.91 (0.05) &0.85 (0.06) & \textbf{0.59} (0.12) &1.01 (0.10)\\
 Equities & 0.87 (0.05) &0.97 (0.03) &0.89 (0.05) & \textbf{0.73} (0.09) &1.18 (0.27)\\
 \hline
 \end{tabular}
 \end{table}

Table \ref{tab:linkage} presents the stability of dendrograms generated using hierarchical clustering with various linkage methods across different data sets: generated data, BRCA and equities. The table displays the normalized distance $d_{\text{coph}}^N$ along with its standard deviation in parentheses. The best results in each row are highlighted in bold.

 When the covariance matrix has a block-diagonal structure, as in the generated data, single linkage (SL) is clearly the most stable method, exhibiting the lowest normalized distance  with a relatively low standard deviation. This indicates that single linkage is less sensitive to small perturbations in the data, maintaining consistent clustering structures.

 For the real data sets, the conclusion holds consistently. In the BRCA data, single linkage again demonstrates superior stability, compared to other methods, suggesting its robustness in clustering biological data where sample variability is often high. Similarly, in the equities data set, single linkage achieves the best stability, indicating its effectiveness in financial data clustering, which often involves high-dimensional and noisy data.

 Analyzing the methods in order of stability across all three data sets, single linkage (SL) is the most stable, followed by average linkage (AL), McQuitty linkage (ML), complete linkage (CL), and Ward's linkage (WL). The higher values for average, McQuitty, complete, and Ward's linkages suggest these methods are more sensitive to data perturbations, resulting in less stable dendrograms. Ward's linkage, in particular, shows the highest  values and standard deviations, indicating it is the least stable method in this context.
 
These results underscore the importance of selecting an appropriate linkage method for hierarchical clustering, especially when stability is a critical concern. Single linkage's robustness across different data types suggests it as a preferable choice for ensuring consistent clustering outcomes in practical applications.

\subsection{Stability of a clustering}
In this section, the dendrogram is cut to focus on clustering. Given that we are considering Gaussian Graphical Models, we can recast this task as a model selection problem. We evaluate two model selection criteria: the Bayesian Information Criterion (BIC) \cite{Schwarz1978} and the slope heuristic (SH) \cite{BirgeMassart2001, Baudry2012}. Additionally, when knowing the ground truth on generated data, we consider the model with 2 clusters, the true number of clusters  $K$ (for generated data), and twice the true number of clusters, $2K$.

Table \ref{tab:ari} displays the Adjusted Rand Index (ARI) \cite{Rand} between clusters derived from hierarchical clustering, cut according to different model selection criteria. The ARI measures the similarity between two partitions, with an ARI of 1 indicating a perfect match. We evaluate the following linkage methods: single linkage (SL), average linkage (AL), complete linkage (CL), McQuitty linkage (ML), and Ward's linkage (WL).

\begin{table}[t]
 \caption{ARI between the clusters get by hierarchical clustering using several linkages and several model selection criterion on generated data, BRCA and equities. Best results are bolded. }
 \label{tab:ari}
\centering
\begin{tabular}{c|c|ccccc}
\hline
Data&criterion&single &average& complete& ward &mcquitty\\ \hline
Generated data&K & 0.19& 0.45 & 0.33& \textbf{0.65} & 0.45\\
&2K & \textbf{0.85} & 0.8 & 0.69& 0.78 & 0.81 \\
\cline{2-7}
&2 & 0.07 & 0.02 & 0.02 & \textbf{0.40} & 0.02\\

&SH & \textbf{0.80} & 0.65 & 0.66& 0.68 & 0.68\\
&BIC & 0.10 & 0.05 & 0.03 & \textbf{0.47} & 0.05\\\hline
BRCA & 2 &0.09 & 0.10& 0.02& \textbf{0.78}& 0.05\\
& SH & \textbf{0.57} & 0.49 & 0.32 &0.37 & 0.44 \\
& BIC & \textbf{0.39} & 0.26 & 0.11 &0.27 & 0.17\\ \hline
Equities & 2 &0.00 & -0.01 & 0.04& \textbf{0.61}& 0.00\\
& SH & \textbf{0.36} & 0.35 & 0.26 &0.22 & 0.31 \\
& BIC & \textbf{0.21} & 0.18 & 0.09 &0.15 & 0.16 \\\hline
\end{tabular}
 \end{table}

For the generated data, the single linkage (SL) and Ward's linkage (WL) show the best performance. Although Ward's linkage performs poorly in terms of distance between dendograms, as seen in Section \ref{hierarchicalClustering}, it performs well when considering clustering, likely due to high variability in early merges but more stability with larger clusters. Notably, single linkage combined with $2K$ clusters achieves the highest ARI, indicating that considering a higher number of clusters enhances stability. Among non-oracle methods, the slope heuristic (SH) combined with single linkage performs the best. Generally, SH provides stable results, whereas BIC performs poorly.

For the real data sets (BRCA and equities), the conclusions are similar: single linkage (SL) and Ward's linkage (WL) are the most competitive. Ward's linkage with 2 clusters is the most stable for both data sets, but not sparse, followed closely by single linkage combined with the slope heuristic (SH).

Table \ref{tab:K} displays the number of clusters selected on average by the slope heuristic and BIC for each dataset. The number of clusters selected by BIC is generally smaller, often underestimating the true number in the simulated data, while the slope heuristic tends to overestimate the number of clusters. Single linkage tends to select a higher number of clusters, contributing to its stability. As with the  $2K$ scenario, the slope heuristic's tendency to overestimate the number of clusters generally contributes to greater stability.

\begin{table}[t]
 \caption{Number of clusters selected (in mean) for several linkages and several model selection criterion on generated data (100 variables, 15 groups), BRCA (200 variables) and equities (100 variables).}
 \label{tab:K}
\centering
\begin{tabular}{c|c|ccccc}
\hline
Data&criterion&single &average& complete& ward &mcquitty\\ \hline
simulated&SH & 25 & 22 & 28 & 22 & 24 \\
&BIC & 10 & 2 & 2 & 5 & 3 \\\hline
equities & SH &155 & 100& 70&60& 90\\
& BIC &105 & 40 & 6 &5 & 28\\\hline
BRCA & SH &112 & 41 & 18 &18 &28\\
& BIC &88 &14 & 5 &7 & 9\\\hline
\end{tabular}
 \end{table}

These results indicate that the choice of linkage method and model selection criterion significantly impacts the stability and accuracy of clustering. Single linkage and the slope heuristic generally provide the most stable results across different data sets and scenarios.

\subsection{Stability of inferred networks}
In this section, we evaluate the stability of  networks inferred by classical methods. 
Stability is assessed using the \emph{normalized Hamming distance} between two graphs $G_1$ and $G_2$, with respective adjacency matrices $A_1$ and $A_2$. The normalized Hamming distance is defined as
\begin{align}
d_H(G_1,G_2) &= \frac{2 \Vert A_1-A_2\Vert _1}{\Vert A_1\Vert_1 + \Vert A_2 \Vert_1}.
\label{HammingDistance}
\end{align}
This metric provides a measure of the difference between two graphs, normalized by the total number of edges in both graphs.
Additionally, we report the density of the inferred graphs, which is the proportion of nonzero coefficients in the adjacency matrix, and the CPU time required for the computations.

For the generated data, we further calibrate the estimation methods using several performance metrics: sensitivity TP/(TP+FN), specificity TN/(TN+FP), precision TP/(TP+FP), and false discovery rate (FDR) FP/(TP+FP), where TP denotes  true positive, FN denotes false negative, FP denotes false positives, and TN denotes true negatives. 
 While these metrics do not directly relate to stability, they help identify methods that infer networks close to the true structure.
One prefers a high precision, recall and specificity, a density close to 0.03 on the generated data (the value on the true graph), a low normalized Hamming distance and a low CPU time.


We compare the following strategies, based on or extended from the Graphical Lasso:

\begin{itemize}
\item One-step Graphical Lasso  methods:
\begin{itemize}
    \item BIC: Regularization parameter selected using Bayesian Information Criterion \cite{Schwarz1978}.
\item EBIC: Extended Bayesian Information Criterion with 
$\gamma=0.5$ \cite{EBIC}.
\item STARS: Stability Approach to Regularization Selection \cite{STARS}.
\item ESCV: Extended Stability Criterion Validation \cite{LimYu}.
\end{itemize}
\item Stabilized methods based on the one-step Graphical Lasso: 
\begin{itemize}
    \item BoLasso (BL): Bootstrap Lasso \cite{bachBoLasso}, with regularization parameter fixed by cross-validation, using bootstrap over $m=100$ samples of size $n$ subsampled from the observations with replacement
    \item 
Stability Selection (SS): \cite{Meinshausen2006}, with $m=100$ samples of size $n/2$  subsampled without replacement, and regularization parameter selected such that $\sqrt{0.8p}$ 
  variables are chosen.
\end{itemize}
\item Two-step Graphical Lasso  methods:
\begin{itemize}
    \item Single Linkage: As highlighted in theory and practice, cut with the slope heuristic, with regularization parameter selected by BIC, STARS, ESCV, and BoLasso within each module.
\item CGL: Average linkage with 2 clusters, where in each module the sparser model is selected \cite{Tan2015}.
\end{itemize}
\end{itemize}

Note that Stability Selection was not run in the two-step Graphical Lasso methods due to high computational cost.

 \subsubsection{Results on simulated dataset}

 Table \ref{tab:simu} details the performance of various methods in inferring from simulated data, 
  focusing on metrics such as normalized Hamming distance, graph density, and CPU time.

 \begin{table}[t]
 \caption{Performance on simulated dataset. We compare the performance in estimation (evaluated by the precision, Prec, the recall, Recall, and the specificity, Spec, and the density of the graph; the performance in stability (evaluated by the normalized hamming distance) and the computation time (evaluated by the CPU time). Best scores are bolded. }
 \label{tab:simu}
\centering
\begin{tabular}{c|cccccc}
 \hline
1 step & BIC & EBIC & STARS & ESCV & BL & SS \\ 
 \hline
Dens & 0.14 & \textbf{0.03} & 0.06 & 0.00 & \textbf{0.03} & 0.01 \\ 
 & \textit{0.04} & \textit{0.01} & \textit{0.00} & \textit{0.00} & \textit{0.00} & \textit{0.00} \\ 
 Hamm & 0.37 & 0.04 & 0.12 & \textbf{0.00} & \textbf{0.01} & \textbf{0.00} \\ 
 & \textit{0.07} & \textit{0.02} & \textit{0.01} & \textit{0.00} & \textit{0.00} & \textit{0.00} \\
 CPU & 16 & 16 & 386 & 235 & 323 & 1681 \\ 
 \hline
2 steps & SL-SH$_{\text{BIC}}$ & SL-SH$_{\text{EBIC}}$ & SL-SH$_{\text{STARS}}$ & SL-SH$_{\text{ESCV}}$ & SL-SH$_{\text{BL}}$ & AL-2$_{\text{sparse}}$ \\ 
 \hline
Dens & 0.04 & 0.04 & 0.01 & 0.01 & 0.02 & 0.04 \\ 
 & \textit{0.00} & \textit{0.00} & \textit{0.00} & \textit{0.00} & \textit{0.00} & \textit{0.00} \\ 
 Hamm & 0.03 & 0.06 & 0.02 & 0.05 & \textbf{0.01} & 0.06 \\ 
 & \textit{0.01} & \textit{0.01} & \textit{0.01} & \textit{0.01} & \textit{0.00} & \textit{0.01} \\ 
 CPU & 12 & 8 & 164 & 102 & 1334 & 11 \\
 \hline 
 \end{tabular}

 \bigskip

 \begin{tabular}{c|cccccc}
 \hline
1 step & BIC & EBIC & STARS & ESCV & BL & SS \\ 
 \hline
Prec & 0.39 & 0.93 & 0.67 & \textbf{1.00} & \textbf{1.00} & \textbf{1.00} \\ 
 & \textit{0.08} & \textit{0.04} & \textit{0.02} & \textit{0.00} & \textit{0.00} & \textit{0.00} \\ 
 Recall & \textbf{0.76} & 0.60 & 0.70 & 0.26 & 0.61 & 0.44 \\ 
 & \textit{0.03} & \textit{0.15} & \textit{0.01} & \textit{0.00} & \textit{0.01} & \textit{0.01} \\
 Spec & 0.89 & \textbf{1.00} & 0.97 & \textbf{1.00} & \textbf{1.00} & \textbf{1.00} \\ 
 & 0.04 & 0.00 & 0.00 & 0.00 & 0.00 & 0.00 \\ 
 \hline
2 steps & SL-SH$_{\text{BIC}}$ & SL-SH$_{\text{EBIC}}$ & SL-SH$_{\text{STARS}}$ & SL-SH$_{\text{ESCV}}$ & SL-SH$_{\text{BL}}$ & AL-2$_{\text{sparse}}$ \\ 
 \hline
Prec & 0.94 & 0.94 & 0.99 & 0.99 & \textbf{1.00} & 0.83 \\ 
 & \textit{0.05} & \textit{0.05} & \textit{0.01} & \textit{0.01} & \textit{0.00} & \textit{0.03} \\ 
 Recall & 0.69 & 0.69 & 0.39 & 0.37 & 0.57 & 0.66 \\ 
 & \textit{0.02} & \textit{0.02} & \textit{0.06}& \textit{0.04} & \textit{0.02} & \textit{0.02} \\ 
 Spec & \textbf{1.00} & \textbf{1.00} & \textbf{1.00} & \textbf{1.00} & \textbf{1.00} & 0.99 \\
 & \textit{0.00} & \textit{0.00} & \textit{0.00} & \textit{0.00} & \textit{0.00} & \textit{0.00} \\ 
 \hline
 \end{tabular}
 \end{table}
 
 In the simulated data context, the single linkage method paired with ESCV achieves the highest stability as indicated by the normalized Hamming distance. This method is closely followed by single linkage with BoLasso, which excels in graph density. These approaches, especially those utilizing single linkage, outperform both one-step methods and CGL in stability.
Important Observation: Graphical Lasso using ESCV often appears very stable (Hamming distance is zero) but infers an empty network (density is zero), making it uninteresting.

Graphical Lasso methods using BIC and STARS do not perform well in estimation, producing overly dense graphs and lacking stability. Among the one-step methods, Graphical Lasso with EBIC, BoLasso, and Stability Selection demonstrate better performance in both estimation quality and stability, albeit with slower computation times due to bootstrap procedures.

Two-step methods generally show improved performance with better estimation accuracy and increased stability. This enhancement is attributed to the block-diagonal network structure inherent in the data generation process, which aligns well with the decomposition strategy used in these methods. While these approaches generally require more computation time, they offer superior stability, particularly evident in methods employing single linkage (SL-SH).

In summary, EBIC emerges as the standout among one-step methods, balancing density, stability, and computational efficiency. Two-step methods, particularly those leveraging single linkage, enhance stability by effectively utilizing the network structure. However, the overall choice of method should consider a trade-off between estimation quality, stability, and computational demands based on specific application requirements.

 \begin{table}[t]
 \caption{Performance on BRCA. We compare the density, the performance in stability (evaluated by the normalized Hamming distance Hamm) and the computation time (evaluated by the CPU time).}
 \label{tab:real}
\centering
\begin{tabular}{c|cccccc}
 \hline 
1 step & BIC & EBIC & STARS & ESCV & BL & SS \\ 
 \hline
Dens & 0.05 & 0.00 & 0.09 & 0.00& 0.01 & 0.00 \\ 
& \textit{0.07} & \textit{0.00} & \textit{0.01} & \textit{0.00} & \textit{0.00} & \textit{0.00} \\
Hamm& 0.19 & 0.00& 0.20 & 0.00 & 0.02 & 0.01 \\ 
 &\textit{ 0.16} & \textit{0.00} & \textit{0.01} & \textit{0.00} & \textit{0.00} & \textit{0.00} \\
 CPU& 73 & 73& 1621 & 971&1071 & 7919 \\ \hline
2 steps & SL-SH$_{\text{BIC}}$ & SL-SH$_{\text{EBIC}}$ & SL-SH$_{\text{STARS}}$ & SL-SH$_{\text{ESCV}}$ & SL-SH$_{\text{BL}}$ & AL-2$_{\text{sparse}}$ \\ 
 \hline
Dens & 0.02& 0.02 & 0.01 & 0 & 0.01 & 0.26 \\ 
 &\textit{0.01} &\textit{0} &\textit{0} &\textit{0} &\textit{0} &\textit{0.03}\\ 
 Hamm & 0.04 & 0.04 & 0.02 & 0 & 0.01 & 0.68 \\ 
 &\textit{0.01} &\textit{0.01} &\textit{0.01} &\textit{0} &\textit{0} &\textit{0.04}\\ 
 CPU & 118 & 14 & 250 & 560 & 2372 & 2 \\ \hline
 \end{tabular}
 \end{table}

 \begin{table}[t]
 \caption{Performance on equities. We compare the density, the performance in stability (evaluated by the normalized Hamming distance Hamm) and the computation time (evaluated by the CPU time). }
 \label{tab:real:2}
\centering
 \begin{tabular}{c|cccccc}
 \hline
 1 step & BIC & EBIC & STARS & ESCV & BL & SS \\ 
 \hline 
 Dens & 0.00 & 0.00 & 0.07 & 0.00 & 0.00 & 0.00 \\ 
 & \textit{0.00} & \textit{0.00} & \textit{0.01} & \textit{0.00} & \textit{0.00} &\textit{ 0.00} \\ 
 Hamm & 0.00 & 0.00 & 0.20 & 0.00 & 0.01 & 0.01 \\ 
 & \textit{0.00} & \textit{0.00} & \textit{0.01} & \textit{0.00} & \textit{0.00} & \textit{0.00} \\ 
CPU & 67 & 68 & 1535 & 917 & 990&7690 \\ \hline
2 steps & SL-SH$_{\text{BIC}}$ & SL-SH$_{\text{EBIC}}$ & SL-SH$_{\text{STARS}}$ & SL-SH$_{\text{ESCV}}$ & SL-SH$_{\text{BL}}$ & AL-2$_{\text{sparse}}$ \\ 
 \hline
Dens & 0.01& 0.01 & 0 & 0 & 0.01 & 0.72 \\ 
 &\textit{0.00} &\textit{0} &\textit{0} &\textit{0} &\textit{0} &\textit{0.01}\\ 
 Hamm & 0.03 & 0.05 & 0.01 & 0.01 & 0.03 & 0.79 \\ 
 &\textit{0.01} &\textit{0.01} &\textit{0.01} &\textit{0.01} &\textit{0.01} &\textit{0.02}\\ 
 CPU & 80 & 6 & 112 & 71 & 651 & 187 \\ \hline
\end{tabular}
 \end{table}

\subsubsection{Real data analysis}

 Tables \ref{tab:real} and \ref{tab:real:2} present the performance on the BRCA and Equities datasets, respectively, in terms of density, normalized Hamming distance (Hamm), and computation time (CPU time in seconds).

\textbf{One-step Methods:}
ESCV and EBIC  show very high stability (Hamming distance is zero) but infer empty networks (density is zero) on both the BRCA and Equities datasets. While stable, they lack practical utility due to this issue.
STARS outperforms BIC and BoLasso in terms of network estimation but comes with significantly higher computational costs.
BoLasso and Stability Selection (SS) show promise in both stability and network estimation quality. However, they are computationally intensive, especially Stability Selection.

\textbf{Two-step Methods:}
Single Linkage with SH demonstrates notable improvements in stability and estimation performance compared to one-step methods. It effectively leverages the block-diagonal network structure present in the generated data.
CGL is generally unstable, confirming theoretical expectations about the limitations of average linkage methods in this context.
Overall, two-step methods improve network estimation and stability across both datasets. They mitigate the limitations observed in one-step methods, particularly in capturing the block-diagonal structure of the generated data.

In conclusion, while one-step methods like STARS show competitive performance, especially in terms of estimation accuracy, two-step methods, particularly those utilizing Single Linkage with appropriate selection criteria, offer superior stability and estimation quality, albeit at increased computational costs. These findings underscore the importance of method selection based on both performance metrics and computational feasibility in practical applications of graphical model inference.

 \section{Discussion and conclusion}
In this paper, we propose an analysis of stability for several network inference methods, with a focus on hierarchical clustering methods using different linkages, influenced by the decomposition of the graphical lasso into two steps. Our study highlights the potential of single linkage in scenarios with a modular structure, challenging conventional wisdom and opening new avenues for stable network inference methods.

Contrary to the common advice to avoid single linkage due to its chaining property \cite{hastie01statisticallearning}, our results demonstrate that single linkage is more stable than other methods when a modular structure is present. This finding is supported by both theoretical analysis and practical experiments.

While our theoretical results are robust, we were unable to provide a complete proof of stability for the full method combining single linkage with any model selection criterion, and particularly considering the slope heuristic (SL+SH). This challenge arises from the complexity involved in accounting for model selection within the stability framework. Addressing this limitation remains an open question and a promising direction for future research.

\bibliographystyle{chicago}
\bibliography{biblioFINALE}

\end{document}